\documentclass[11pt]{amsart}

\usepackage{amscd,amssymb,amsmath,graphicx,verbatim}
\usepackage{mathrsfs}  
\usepackage{wasysym}
\usepackage{hyperref}
\usepackage[TS1,OT1,T1]{fontenc}
\usepackage{lscape} 
\usepackage{fullpage} 
\usepackage{pslatex} 
\usepackage{tikz}
\usetikzlibrary{shapes,arrows}
\usepackage[all]{xy}
\newtheorem{theorem}{Theorem}[section]
\newtheorem*{theorem*}{Theorem}
\newtheorem{lemma}[theorem]{Lemma}
\newtheorem{corollary}[theorem]{Corollary}
\newtheorem{proposition}[theorem]{Proposition}

\theoremstyle{definition}
\newtheorem{definition}[theorem]{Definition}

\newtheorem{example}[theorem]{Example}

\theoremstyle{remark}
\newtheorem{remark}[theorem]{Remark}

\DeclareMathOperator{\thick}{thick}
\DeclareMathOperator{\Hom}{Hom}
\DeclareMathOperator{\Ext}{Ext}

\DeclareMathOperator*{\Mod}{\!-\mathsf{Mod}}

\DeclareMathOperator*{\proj}{\!-\mathsf{proj}}
\DeclareMathOperator*{\Proj}{\!-\mathsf{Proj}}
\DeclareMathOperator*{\Inj}{\!-\mathsf{Inj}}

\newcommand{\Add}{\ensuremath{\mathsf{Add}}}
\newcommand{\Prod}{\ensuremath{\mathsf{Prod}}}

\newcommand{\Tcal}{\ensuremath{\mathcal{T}}}

\newcommand{\Fcal}{\ensuremath{\mathcal{F}}}
\newcommand{\Ccal}{\ensuremath{\mathcal{C}}}

\newcommand{\Acal}{\ensuremath{\mathcal{A}}}

\newcommand{\Wcal}{\ensuremath{\mathcal{W}}}
\newcommand{\Hcal}{\ensuremath{\mathcal{H}}}
\newcommand{\Scal}{\ensuremath{\mathcal{S}}}

\newcommand{\Ucal}{\ensuremath{\mathcal{U}}}
\newcommand{\Vcal}{\ensuremath{\mathcal{V}}}

\newcommand{\K}{\mathsf{K}}
\newcommand{\C}{\mathsf{C}}
\newcommand{\D}{\mathsf{D}}
\newcommand{\Kbb}{\mathbb{K}}

\newcommand{\ra}{\rightarrow}

\DeclareMathOperator{\Rep}{Rep}

\numberwithin{equation}{section}

\begin{document}
\title{Silting and cosilting classes in derived categories}
\author{Frederik Marks, Jorge Vit{\'o}ria}
\address{Frederik Marks, Institut f\"ur Algebra und Zahlentheorie, Universit\"at Stuttgart, Pfaffenwaldring 57, 70569 Stuttgart, Germany}
\email{marks@mathematik.uni-stuttgart.de}
\address{Jorge Vit\'oria, Department of Mathematics, City, University of London, Northampton Square, London EC1V 0HB, United Kingdom}
\email{jorge.vitoria@city.ac.uk}
\keywords{torsion pair, t-structure, co-t-structure, silting complex, cosilting complex, derived category}
\subjclass[2010]{16E35,18E30,18E40}
\thanks{The authors are grateful to Lidia Angeleri H\"ugel for numerous stimulating discussions on the ideas presented in this article. The first named author was supported by a grant within the DAAD P.R.I.M.E. program. The second named author acknowledges support from the Department of Computer Sciences of the University of Verona in the earlier part of this project, and from the Engineering and Physical Sciences Research Council of the United Kingdom, grant number EP/N016505/1, in the later part of this project. Finally, the authors acknowledge funding from the Project ``Ricerca di Base 2015'' of the University of Verona}
\maketitle


\begin{abstract}
An important result in tilting theory states that a class of modules over a ring is a tilting class if and only if it is the Ext-orthogonal class to a set of compact modules of bounded projective dimension. Moreover, cotilting classes are precisely the resolving and definable subcategories of the module category whose Ext-orthogonal class has bounded injective dimension.
 
In this article, we prove a derived counterpart of the statements above in the context of silting theory.
Silting and cosilting complexes in the derived category of a ring generalise tilting and cotilting modules. They give rise to subcategories of the derived category, called silting and cosilting classes, which are part of both a t-structure and a co-t-structure. We characterise these subcategories: silting classes are precisely those which are intermediate and Ext-orthogonal classes to a set of compact objects, and cosilting classes are precisely the cosuspended, definable and co-intermediate subcategories of the derived category.
\end{abstract}


\section{introduction}
Silting and cosilting complexes, as introduced in \cite{AMV1}, \cite{Wei0} and \cite{WeiZhang}, can be understood as derived analogues of tilting and cotilting modules. This paper reinforces such a perspective through the torsion pairs naturally associated to these complexes, namely t-structures and co-t-structures. 
To every silting or cosilting complex we associate both a t-structure and a co-t-structure which turn out to be adjacent, that is, the \textit{torsion} class of one of the pairs turns out to be the \textit{torsionfree} class of the other. Such triples of subcategories are instances of so-called torsion-torsionfree (TTF) triples and the subcategory that links the two pairs will be referred to as the corresponding silting, respectively cosilting, class. 
For silting complexes, these triples were first observed in \cite{KY} in the bounded derived category of a finite-dimensional algebra of finite global dimension, and they were further generalised in \cite{AMV1} to unbounded derived categories of rings. For cosilting complexes, we will prove in this paper the existence of such triples, using results from \cite{AMV3}, \cite{Bondarko} and \cite{WeiZhang}. 

Our main aim is to describe silting and cosilting classes in the derived category of a ring from a recurrent point of view in (classical) tilting theory. In a module category it is well known that a class of modules is:
\begin{itemize}
\item a tilting class if and only if it is the right Ext-orthogonal class to a set of compact modules of bounded projective dimension (see \cite{BS} and the references therein);
\item a cotilting class if and only if it is resolving, definable and such that the right Ext-orthogonal class to it has bounded injective dimension (see \cite[Theorem 6.1]{Ba1}).
\end{itemize}

The first statement is known as the finite-type characterisation of tilting classes and it can be restated by saying that there is a bijection between resolving subcategories of compact modules of bounded projective dimension and tilting modules up to equivalence (see also \cite{AHT}). Note that cotilting modules are generally not analogously determined by a set of compact modules. They are, however, always pure-injective (\cite{Ba0}, \cite{Stov}).

Our main theorem generalises the results above to (co)silting classes in the derived category of a ring.\\

\begin{theorem*}
Let $A$ be a ring, $\D(A)$ the derived category of $A\Mod$ and $\Vcal$ a full subcategory of $\D(A)$. 
\begin{enumerate}
\item $\Vcal$ is a silting class if and only if $\Vcal$ is intermediate and $\Vcal=\Scal^{\perp_{>0}}$ for a set $\Scal$ of compact objects.
\item $\Vcal$ is a cosilting class if and only if $\Vcal$ is cosuspended, co-intermediate and definable.
\end{enumerate}
Moreover, the first statement induces a bijection between co-intermediate and cosuspended subcategories of $\K^b(A\proj)$ and silting complexes up to equivalence.
\end{theorem*}

This statement summarises Theorems \ref{thm silting class} and \ref{thm cosilting class}. For the terminology used, we refer to the relevant sections. However, the parallel with the results concerning tilting and cotilting modules is evident.

The proof of our main theorem involves some module categories built from the category of complexes. In fact, we reduce (co)silting problems in the derived category to (co)tilting problems in these module categories. This allows us to use the finite-type characterisation of tilting classes to conclude the compact generation of silting co-t-structures. Moreover, we use the fact that cotilting modules are pure-injective to conclude that so are cosilting complexes, from which we deduce the definability of cosilting classes.

Finally, note that for the special case of two-term silting and cosilting complexes, the theorem above specialises to certain classification results recently obtained in the context of silting and cosilting modules. More precisely, part (1) translates to saying that a torsion class in the module category arises from a silting module if and only if it is \textit{divisible} with respect to a set of maps between finitely generated projective modules (see \cite[Theorem 6.3]{MaSt}). Part (2) of the theorem is, in this context, equivalent to stating that a torsionfree class in the module category arises from a cosilting module if and only if it is definable (see \cite[Corollary 3.9]{A} and the references therein).

\smallskip

The structure of the paper is as follows. In Section 2, we set up the aforementioned module categories built from the category of complexes, and we show a useful correspondence between (co)silting complexes in the derived category and certain (co)tilting modules over this new ring. In Section 3, which is subdivided into two subsections on silting and cosilting classes, we prove the main theorem of the paper.

\medskip

\noindent\textbf{Notation.}
All subcategories considered are strict and full. Throughout, let $A$ be a ring. The category of left $A$-modules is denoted by $A\Mod$, its subcategory of projective modules by $A\Proj$, its subcategory of finitely generated projective modules by $A\proj$ and its subcategory of injective modules by $A\Inj$. We write $A^+$ for the injective cogenerator $\Hom_\mathbb{Z}(A,\mathbb{Q}/\mathbb{Z})$ of $A\Mod$. If $\Ccal$ is subcategory of $A\Mod$, we denote by $\Ccal^{\perp}$ (respectively, $\Ccal^{\perp_1}$) the subcategory of $A\Mod$ formed by all objects $M$ such that $\Ext_A^j(X,M)=0$ for all $X$ in $\Ccal$ and for all $j>0$ (respectively, for $j=1$). Similarly, one defines ${}^\perp\Ccal$ (respectively, ${}^{\perp_1}\Ccal$). 

For an additive category $\Acal$, we denote by $\C(\Acal)$ and $\K^*(\Acal)$ ($*=b,\emptyset$) the corresponding category of complexes and the (bounded) homotopy category, respectively. If $\Acal=A\Mod$, we simply write $\C(A)$ and $\K^*(A)$. We also denote by $\D(A)$ the derived category of $A\Mod$. If $X$ is an object of an additive category $\Acal$, we denote by $\Add(X)$ (respectively, $\Prod(X)$) the smallest subcategory of $\Acal$ containing $X$ and closed under coproducts (respectively, products) and summands. 

If $\Ccal$ is a subcategory of a triangulated category $\Tcal$, we denote by $\mathsf{thick}(\Ccal)$ the smallest triangulated subcategory of $\Tcal$ closed under summands and containing $\Ccal$. Given a set of integers $I$ (which is often expressed by symbols such as $>n$, $<n$, $\geq n$, $\leq n$, $\neq n$, or just $n$, with the obvious associated meaning) we define the orthogonal classes ${}^{\perp_I}\Ccal$ and ${\Ccal}^{\perp_I}$ as follows.
$${}^{\perp_I}\Ccal:=\{Y\in \Tcal\mid \Hom_\Tcal(Y,X[i])=0, \forall X\in\Ccal, \forall i\in I\}\ \ \ \ {\Ccal}^{\perp_I}:=\{Y\in \Tcal\mid \Hom_\Tcal(X,Y[i])=0, \forall X\in\Ccal, \forall i\in I\}$$


\section{Complexes seen as modules}

\subsection{The category of complexes}
We begin by reviewing some homological aspects of the category of complexes of left $A$-modules. There are two natural exact structures on $\C(A)$:
\begin{itemize} 
\item an \textbf{abelian} structure, where conflations are short exact sequences of complexes. The notation $\C(A)$ will stand for the category endowed with this exact structure. The abelian category $\C(A)$ has enough projective and enough injective objects. The projective objects in $\C(A)$ are precisely those which lie in the additive closure of all complexes $(X^i,d^i)_{i\in\mathbb{Z}}$ for which there is some $n$ in $\mathbb{Z}$ such that $X^i=0$ for all $i\neq n,n+1$ and $X^n=X^{n+1}$, where $X^n$ is a projective $A$-module and $d^n$ is an isomorphism. In other words, a complex is projective if and only if it is a split exact complex of projective $A$-modules. Dually, the injective objects in $\C(A)$ are the split exact complexes of injective $A$-modules (see also \cite[Exercise 2.2.1]{Weibel}).
\medskip
\item a \textbf{semi-split} exact structure, where conflations are semi-split short exact sequences of complexes, i.e. exact sequences of complexes that split componentwise. We will write $\C_{s}(A)$ whenever the category is endowed with this exact structure. Note that $\C_{s}(A)$ is a Frobenius exact category, whose projective (and injective) objects are those in the additive closure of all complexes $(X^i,d^i)_{i\in\mathbb{Z}}$ for which there is some $n$ in $\mathbb{Z}$ such that $X^i=0$ for all $i\neq n,n+1$, $X^n=X^{n+1}$ and $d^n$ is an isomorphism. In other words, a complex is projective (and injective) in $\C_s(A)$ if and only if it is a split exact complex of $A$-modules. The stable category of the Frobenius exact category $\C_s(A)$ turns out to be triangle equivalent to the homotopy category $\K(A)$  (see also \cite[Subsection I.3.2]{Happel}).
\end{itemize}

Here is a useful proposition regarding Ext-groups in $\mathsf{C}(A)$.
\begin{proposition}\label{computing Ext}
Let $X$ and $Y$ be objects in $\mathsf{C}(A)$. If $X$ lies in $\C(A\Proj)$ or $Y$ lies in $\C(A\Inj)$, then $$\Ext^j_{\C(A)}(X,Y)\cong \Hom_{\D(A)}(X,Y[j]), \ \forall j>0.$$
\end{proposition}
\begin{proof}
Let $X$ be in $\C(A\Proj)$ and $Y$ be in $\C(A)$ (the arguments for when $Y$ lies in $\C(A\Inj)$ and $X$ is any complex in $\C(A)$ are dual). Since the stable category of $\C_s(A)$ is triangle equivalent to the homotopy category $\K(A)$, and since $X$ lies in $\C(A\Proj)$, we have the following sequence of natural isomorphisms for all $j>0$
$$\Ext_{\C_s(A)}^j(X,Y)\cong \Hom_{\K(A)}(X,Y[j])\cong \Hom_{\D(A)}(X,Y[j]).$$
We show by induction that the natural embedding of $\Ext_{\C_s(A)}^j(X,Y)$ into $\Ext_{\C(A)}^j(X,Y)$ is an isomorphism.
For $j=1$, any short exact sequence of complexes in $\Ext^1_{\C(A)}(X,Y)$ splits componentwise and, hence, it lies naturally in $\Ext^1_{\C_s(A)}(X,Y)$. 
Suppose now that the statement holds for some $j>0$. We use dimension shifting to show the inductive step. Indeed, consider a projective object $P$ of $\C(A)$ yielding a (semi-split) short exact sequence of the form
$$\xymatrix{0\ar[r]& \Omega_X\ar[r]& P\ar[r]& X\ar[r] &0.}$$
It is then clear that $\Ext^{j+1}_{\C(A)}(X,Y)$ is isomorphic to  $\Ext^{j}_{\C(A)}(\Omega_X,Y)$ with $\Omega_X$ in $\C(A\Proj)$. By induction, this is further isomorphic to $\Ext_{\C_s(A)}^j(\Omega_X,Y)\cong\Ext_{\C_s(A)}^{j+1}(X,Y)$, as wanted.
\end{proof}

\begin{remark}
For two complexes $X$ and $Y$, the group $\Hom_{\D(A)}(X,Y[j])$ is sometimes denoted by $\Ext_{A}^j(X,Y)$. The motivation for this comes from the fact that if $X$ and $Y$ lie in $A\Mod$, then the group of Yoneda extensions indeed coincides with the corresponding $\Hom$-space in the derived category. The above proposition shows that this motivation can be extended to complexes, provided that either $X$ lies in $\C(A\Proj)$ or $Y$ lies in $\C(A\Inj)$. However, it is not true in general that Yoneda extension groups between complexes can be computed in the derived category. Indeed, for example, extension groups between acyclic complexes (or even between complexes homotopically equivalent to the zero complex) may be non-trivial while such objects are isomorphic to zero in the derived category (respectively, in the homotopy category).
\end{remark}


\subsection{The categories $\Rep(AQ_n/I)$}

Fix $n>0$ and let $Q_n$ be the Dynkin quiver
$$\xymatrix{-n+1\ar[rr]^{\alpha_{-n+1}}&&-n+2\ar[rr]^{\alpha_{-n+2}}&&\cdots \ar[rr]^{\alpha_{-2}}&&-1\ar[rr]^{\alpha_{-1}}&&0}.$$
We are interested in the ring $AQ_n/I$ defined to be the quotient of the path algebra $AQ_n$ by the ideal $I$ generated by all paths of length two. Alternatively, we can think of $AQ_n/I$ as the quotient of the lower triangular matrix ring $T_n(A)$ by the two-sided ideal generated by all the elementary matrices $E_{ij}$ with $i-j\ge 2$. 

By $\Rep(AQ_n/I)$ we denote the category of representations of $Q_n$ in $A\Mod$ that are bound by the relations defining $I$.
An object $M$ in $\Rep(AQ_n/I)$ is a sequence of left $A$-modules $M(i)$ together with homomorphisms $M(\alpha_i):M(i)\rightarrow M(i+1)$ such that $M(\alpha_{i+1})M(\alpha_i)=0$. A morphism $f:M\rightarrow N$ in $\Rep(AQ_n/I)$ is a family $(f_i)_{-n+1\leq i\leq 0}$ of homomorphisms of left $A$-modules $f_i:M(i)\rightarrow N(i)$ such that $N(\alpha_i)f_i=f_{i+1}M(\alpha_i)$. By $\Rep_\mathscr{P}(AQ_n/I)$ and $\Rep_\mathscr{I}(AQ_n/I)$ we denote the full subcategories of $\Rep(AQ_n/I)$ consisting of objects $M$ such that all $M(i)$ lie in $A\Proj$ or in $A\Inj$, respectively.
Note that the category $\Rep(AQ_n/I)$ is equivalent to the module category $AQ_n/I\Mod$. We can also think of $\Rep(AQ_n/I)$ as the category of covariant functors from the small category defined by $Q_n$ (with zero relations on the morphisms induced by $I$) to $A\Mod$.

\begin{example}\label{example}
In case $n=2$, the category $\Rep(AQ_2/I)=\Rep(AQ_2)$ is naturally equivalent to the morphism category of $A$, whose objects are maps of $A$-modules and whose morphisms are commutative squares relating them. 
We provide an explicit example. Let $\Kbb$ be a field and choose $A$ to be $\Kbb Q_2$. The structure of the category $A\Mod\cong\Rep(\Kbb Q_2)$ is represented by the following Auslander-Reiten quiver
$${\small
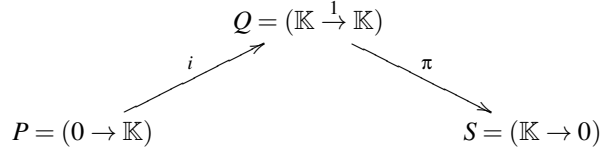
{\xymatrix{& Q=(\Kbb\overset{1}{\ra}\Kbb)\ar[dr]^\pi & \\ P=(0\ra\Kbb)\ar[ur]^i & & S=(\Kbb\ra 0)}}}$$
The algebra $AQ_2$ turns out to be isomorphic to the quotient of the path algebra over $\Kbb$ of the quiver 
$${\small
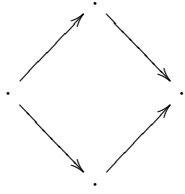
{\xymatrix{& \cdot\ar[dr] & \\ \cdot\ar[ur]\ar[dr] & & \cdot \\ & \cdot\ar[ur] &}}}$$
by the ideal generated by the commutativity relation. The structure of the category $AQ_2\Mod\cong\Rep(AQ_2)$ is represented by the following Auslander-Reiten quiver
$${\small
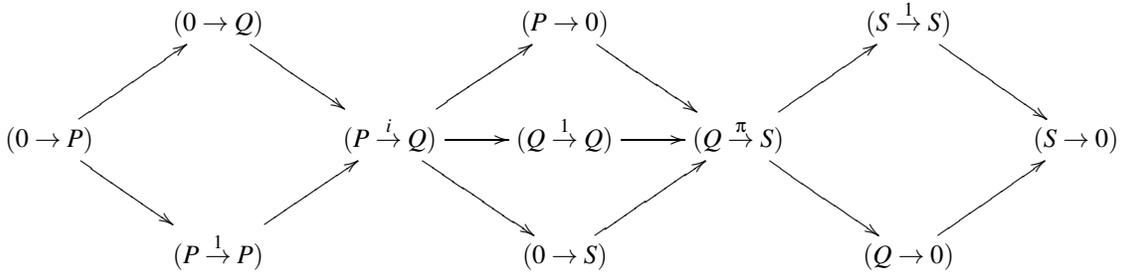
{\xymatrix{& (0\ra Q)\ar[dr] & & (P\ra 0)\ar[dr] & & (S\overset{1}{\ra} S)\ar[dr] & \\ (0\ra P)\ar[dr]\ar[ur] & & (P\overset{i}{\ra} Q)\ar[ur]\ar[r]\ar[dr] & (Q\overset{1}{\ra} Q)\ar[r] & (Q\overset{\pi}{\ra} S)\ar[ur]\ar[dr] & & (S\ra 0)\\ & (P\overset{1}{\ra} P)\ar[ur] & & (0\ra S)\ar[ur] & & (Q\ra 0)\ar[ur] &}}}$$
This example may help to visualise the more technical statements obtained later in this section.
\end{example}

In what follows, we relate representations in $\Rep(AQ_n/I)$ to complexes of $A$-modules. Note that the category $\Rep(AQ_n/I)$ can naturally be identified with a full subcategory of $\C(A)$. Indeed, consider the fully faithful functor $\Theta\colon \Rep(AQ_n/I)\longrightarrow \C(A)$ that maps $M$ in $\Rep(AQ_n/I)$ to the complex $(X^i,d^i)$ with $X^i=M(i)$ for all $i\in\{-n+1,...,0\}$, $d^i=M(\alpha_{i})$ for all $i\in\{-n+1,...,-1\}$ and $X^i=0$ otherwise. We will also be interested in the functor $\Psi:\Rep(AQ_n/I)\longrightarrow \D(A)$ defined to be the composition of $\Theta$ with the canonical functor $\C(A)\longrightarrow \D(A)$.

\begin{lemma}\label{ff homological} 
The full embedding $\Theta$ is homological, i.e. for all objects $M$ and $N$ in $\Rep(AQ_n/I)$ and all $j>0$, the functor $\Theta$ induces an isomorphism 
$$\Ext_{AQ_n/I}^j(M,N)\cong\Ext_{\C(A)}^j(\Theta(M),\Theta(N)).$$
If, moreover, $M$ lies in $\Rep_\mathscr{P}(AQ_n/I)$ or $N$ lies in $\Rep_\mathscr{I}(AQ_n/I)$, then
$$\Ext_{AQ_n/I}^j(M,N)\cong\Hom_{\D(A)}(\Psi(M),\Psi(N)[j]).$$
\end{lemma}

\begin{proof}
Clearly, $\Theta$ is exact and it induces an injective map $\Ext_{AQ_n/I}^j(M,N)\longrightarrow \Ext_{\C(A)}^j(\Theta(M),\Theta(N))$ for any $j>0$ and any $M,N$ in $\Rep(AQ_n/I)$. We observe that this map is also surjective, thus showing that $\Theta$ is homological. Given a Yoneda extension $\eta$ of the form
$$\xymatrix{0\ar[r] & \Theta(N)\ar[r] & E_1\ar[r] & ...\ar[r] & E_j\ar[r] & \Theta(M)\ar[r] & 0}$$
it is enough to show that, up to equivalence in $\Ext_{\C(A)}^j(\Theta(M),\Theta(N))$, 
we can choose the complexes $E_i$ for $i\in\{1,...,j\}$ to only have non-zero components in degrees $-n+1$ to $0$.
Define $\eta^*$ to be the exact sequence of complexes we obtain from $\eta$ by setting $(E_i)^k=0$ for all $k>0$ and $\eta^{**}$ to be the exact sequence we obtain from $\eta^*$ by setting $(E_i)^k=0$ for all $k<-n+1$. Now it is not hard to check that there are maps from the exact sequences $\eta$ and $\eta^{**}$ to the exact sequence $\eta^*$, thus showing that $\eta$ and $\eta^{**}$ represent the same element in $\Ext_{\C(A)}^j(\Theta(M),\Theta(N))$. Finally, the last statement follows from Proposition \ref{computing Ext}.
\end{proof}

We often identify an object $M$ in $\Rep(AQ_n/I)$ with its image under $\Theta$ or $\Psi$. Using the lemma above, we can compute the projective or injective dimension of certain objects in $\Rep(AQ_n/I)$ and, in particular, identify the projective and injective objects there. Given $j\in\{-n+1,...,0\}$ and $X$ in $A\Mod$, we define the following representations $X_j$ and $X^j$ in $\Rep(AQ_n/I)$.

\begin{itemize}
\item If $j\neq 0$ (respectively, $j\neq -n+1$) we define $X_j$ (respectively, $X^j$) to correspond under $\Theta$ to the following complex concentrated in degrees $j$ and $j+1$ (respectively, in degrees $j$ and $j-1$):
$$\xymatrix{\cdots\ar[r]&0\ar[r]&X\ar[r]^1&X\ar[r]&0\ar[r]&\cdots}$$
Note that $X_j=X^{j+1}$ for all $j\in\{-n+1,...,-1\}$. Direct sums of such objects are called \textbf{contractible}.
\item If $j=0$ (respectively, $j=-n+1$) we define $X_j$ (respectively, $X^j$) to correspond to the stalk complex 
$$\xymatrix{\cdots\ar[r]&0\ar[r]&0\ar[r]&X\ar[r]&0\ar[r]&0\ar[r]&\cdots}$$
concentrated in degree $0$ only (respectively, in degree $-n+1$).
\end{itemize}

\begin{lemma}\label{properties CnA}
The following statements hold.
\begin{enumerate}
\item The full subcategories $\Rep_\mathscr{P}(AQ_n/I)$ and $\Rep_\mathscr{I}(AQ_n/I)$ are extension-closed in $\Rep(AQ_n/I)$ and, thus, they are exact subcategories for the inherited exact structure.
\item The projective objects of $\Rep(AQ_n/I)$ lie in $\Rep_\mathscr{P}(AQ_n/I)$. More precisely, $AQ_n/I\Proj$ identifies with the additive closure of the representations $P_j$ where $P$ is a projective $A$-module. A direct sum of objects of the form $P_j$ with $j\neq 0$ is said to be \textbf{contractible projective}.
\item The injective objects of $\Rep(AQ_n/I)$ lie in $\Rep_\mathscr{I}(AQ_n/I)$. More precisely, $AQ_n/I\Inj$ identifies with the closure under products of the representations $E^j$ where $E$ denotes an injective $A$-module. A product of objects of the form $E^j$ with $j\neq -n+1$ is said to be \textbf{contractible injective}.
\item The projective objects of $\Rep_\mathscr{P}(AQ_n/I)$ are precisely the projective representations of $\Rep(AQ_n/I)$. The injective objects of $\Rep_\mathscr{P}(AQ_n/I)$ are given by the contractible projectives and the representations $P^{-n+1}$ for $P$ projective in $A\Mod$.
\item The injective objects of $\Rep_\mathscr{I}(AQ_n/I)$ are precisely the injective representations of $\Rep(AQ_n/I)$. The projective objects of $\Rep_\mathscr{I}(AQ_n/I)$  are given by the contractible injectives and the representations $E_0$ for $E$ injective in $A\Mod$.
\item For $M$ in $\Rep_\mathscr{P}(AQ_n/I)$ without projective summands, the projective dimension of $M$ in $\Rep(AQ_n/I)$ equals $-j$, where $j$ is the smallest integer in $\{-n+1,...,-1\}$ such that $M(j)\neq 0$.
\item For $M$ in $\Rep_\mathscr{I}(AQ_n/I)$ without injective summands, the injective dimension of $X$ in $\Rep(AQ_n/I)$ equals $n-1+j$, where $j$ is the largest integer in $\{-n+2,...,0\}$ such that $M(j)\neq 0$.
\end{enumerate}
\end{lemma}
\begin{proof}
The first statement is immediate. We will prove statements (2), (4) and (6). Statements (3), (5) and (7) can be shown dually. 

(2): It is clear that the projective objects of $\Rep(AQ_n/I)$ must lie in $\Rep_\mathscr{P}(AQ_n/I)$. By Lemma \ref{ff homological}, contractible projective objects are indeed projective in $\Rep(AQ_n/I)$, since, when regarded as complexes, they are  homotopic to zero. Moreover, for any projective $P$ in $A\Mod$, also the representation $P_0$ is projective in $\Rep(AQ_n/I)$, since $\Ext^1_{AQ_n/I}(P_0,-)\cong \Hom_{\D(A)}(\Psi(P_0),\Psi(-)[1])=0$. Now let $M$ be projective in $\Rep(AQ_n/I)$ without contractible summands. We show that $M$ is of the form $P_0$ for some $P$ in $A\Proj$. Suppose that $M(j)\not=0$ for some $j\in\{-n+1,...,-1\}$ and let $M^\prime$ be the representation in $\Rep_\mathscr{P}(AQ_n/I)$ given by $M^\prime(i)=M(i-1)$ and $M^\prime(\alpha_i)=M(\alpha_{i-1})$ for all $i\in\{-n+2,...,0\}$ and $M^\prime(-n+1)=0$. By assumption on $M$, it follows that $\Psi(M)$ and $\Psi(M^\prime)$ are non-zero objects in the derived category $\D(A)$. Furthermore, by construction of $M^\prime$ there is a non-zero morphism in $\Hom_{\D(A)}(\Psi(M),\Psi(M^\prime)[1])$ given by identity maps in each component different from zero. This contradicts our assumption of $M$ being projective. Thus, the projectives in $\Rep(AQ_n/I)$ with no contractible summands are those of the form $P_0$.

(4): Clearly, all projectives in $\Rep(AQ_n/I)$ are projective in $\Rep_\mathscr{P}(AQ_n/I)$ and it follows from the same argument as in (2) that these are the only projective objects of $\Rep_\mathscr{P}(AQ_n/I)$. Using Lemma \ref{ff homological}, it is easy to see that the contractible projective objects are also injective in $\Rep_\mathscr{P}(AQ_n/I)$, and a dual argument to the one used in (2) allows us to see that an injective object of $\Rep_\mathscr{P}(AQ_n/I)$ without contractible summands must be of the form $P^{-n+1}$ for some projective $A$-module $P$.

(6): This follows from Lemma \ref{ff homological} by observing that for $M$ in $\Rep_\mathscr{P}(AQ_n/I)$ and $j$ the smallest integer in $\{-n+1,...,-1\}$ such that $M(j)\neq 0$, we have 
$$\Ext_{AQ_n/I}^{-j+1}(M,N)\cong\Hom_{\D(A)}(\Psi(M),\Psi(N)[-j+1])=0$$ 
for all $N$ in $\Rep(AQ_n/I)$. Moreover, since $M$ has no projective summands, the identity map on $M(j)$ yields a non-zero extension in $\Ext_{AQ_n/I}^{-j}(M,(M(j))_0)$.
\end{proof}

We encourage the reader to revisit Example \ref{example} to test the lemma above in a concrete setting.
The following statement is an immediate consequence regarding global dimension.

\begin{corollary}
If the ring $A$ has left global dimension $d$, then $AQ_n/I$ has left global dimension $d+n-1$. Moreover, if $A$ is of infinite left global dimension, then so is $AQ_n/I$.
\end{corollary}

\begin{proof}
Take an object $N$ in $\Rep(AQ_n/I)$ and choose a projective resolution of minimal length for every $A$-module $N(j)$ with $j\in\{-n+1,...,0\}$. These resolutions give rise to a long exact sequence in $\Rep(AQ_n/I)$
$$\xymatrix{ 0\ar[r] & N_d\ar[r] & N_{d-1}\ar[r] & ...\ar[r] & N_0\ar[r] & N\ar[r] & 0.}$$
By construction, all the $N_i$ lie in $\Rep_\mathscr{P}(AQ_n/I)$ for $i\in\{0,...,d\}$ and, thus, by Lemma \ref{properties CnA}(6), they have projective dimension at most $n-1$. Consequently, the projective dimension of $N$ is bounded by $d+n-1$. Finally, it is easy to check that for any $A$-module $X$ of projective dimension $d$, the representation $X^{-n+1}$ has projective dimension $d+n-1$ in $\Rep(AQ_n/I)$.
\end{proof}

\begin{remark}
We define the infinite quiver $Q_{\infty}^\infty$ as 
$$\xymatrix{\cdots\ar[rr]&&-n\ar[rr]^{\alpha_{-n}}&&-n+1\ar[rr]^{\alpha_{-n+1}}&&\cdots \ar[rr]^{\alpha_{-1}}&&0\ar[rr]^{\alpha_{0}}&&1\ar[rr]&&\cdots}$$
with its two full subquivers $Q_{\infty}$ and $Q^{\infty}$ having a sink (respectively, a source) at vertex $0$.
As before, we can consider associated categories of representations where $\Rep(AQ^\infty_\infty/I)$ is naturally equivalent to $\C(A)$. 
The category $\Rep(AQ_\infty/I)$ can be used to translate the parts of Lemma \ref{properties CnA} concerning $\Rep_\mathscr{P}(AQ_n/I)$ to an infinite context, while the category $\Rep(AQ^\infty/I)$ serves to generalise the observations in Lemma \ref{properties CnA} on $\Rep_\mathscr{I}(AQ_n/I)$. Moreover, note that the categories of representations induced by $Q_{\infty}^\infty$, $Q_{\infty}$ and $Q^\infty$ are Grothendieck abelian categories with enough projectives, but they are no longer module categories of a unital ring. Therefore, we will restrict ourselves in the forthcoming sections to the finite case of $\Rep(AQ_n/I)$.
\end{remark}


\subsection{(Co)Silting complexes and (co)tilting modules}
Let us fix $n>0$. We recall the following definitions.
\begin{definition}
A complex $X$ in $\D(A)$ is said to be 
\begin{itemize}
\item \textbf{silting} if $\Hom_{\D(A)}(X,X^{(J)}[i])=0$ for all $i>0$ and all sets $J$, and $\thick(\Add(X))=\K^b(A\Proj)$; it is moreover said to be \textbf{$n$-silting} if $X$ lies in $\Psi(\Rep_\mathscr{P}(AQ_n/I))$.
\item \textbf{cosilting} if $\Hom_{\D(A)}(X^J,X[i])=0$ for all $i>0$ and all sets $J$, and $\thick(\Prod(X))=\K^b(A\Inj)$; it is moreover said to be \textbf{$n$-cosilting} if $X$ lies in $\Psi(\Rep_\mathscr{I}(AQ_n/I))$.
\end{itemize}
An $A$-module $M$ is said to be
\begin{itemize}
\item $(n-1)$-\textbf{tilting} if it has projective dimension at most $n-1$, $\Ext_A^i(M,M^{(J)})=0$ for all $i>0$ and all sets $J$, and there is an exact sequence of $A$-modules with all $M_i$ in $\Add(M)$
$$\xymatrix{0\ar[r] & A\ar[r] & M_0\ar[r] & M_1\ar[r] & ...\ar[r] & M_{n-1}\ar[r] & 0.}$$
\item  $(n-1)$-\textbf{cotilting} if it has injective dimension at most $n-1$, $\Ext_A^i(M^{J},M)=0$ for all $i>0$ and all sets $J$, and there is an exact sequence of $A$-modules with all $M_i$ in $\Prod(M)$ $$\xymatrix{0\ar[r] & M_{n-1}\ar[r] & ...\ar[r] & M_1\ar[r] & M_0\ar[r] & A^+\ar[r] & 0.}$$
\end{itemize}
We say that two silting (respectively, cosilting) complexes $X$ and $X^\prime$ are \textbf{equivalent} if $\Add(X)=\Add(X')$ (respectively, $\Prod(X)=\Prod(X')$). Similarly, two tilting (respectively, cotilting) modules $M$ and $M^\prime$ are \textbf{equivalent} if $\Add(M)=\Add(M^\prime)$ (respectively, $\Prod(M)=\Prod(M^\prime)$).
\end{definition}

It can be checked that an $A$-module is $(n-1)$-tilting if and only if the corresponding stalk complex concentrated in degree $0$ is $n$-silting. Dually, an $A$-module is $(n-1)$-cotilting if and only if the corresponding stalk complex concentrated in degree $-n+1$ is $n$-cosilting. 

\begin{example}\label{example proj inj}
Consider the ring $AQ_n/I$ and the injective cogenerator $T$ of $\Rep_\mathscr{P}(AQ_n/I)$ given by
the direct sum of all $A^j$ for $j\in\{-n+1,...,0\}$.
It follows from Lemma \ref{properties CnA} that $T$ is an $(n-1)$-tilting module over $AQ_n/I$. There is an associated tilting cotorsion pair in $\Rep(AQ_n/I)$ given by $(^\perp(T^\perp),T^\perp)=(^{\perp_1}(\Scal^{\perp_1}),\Scal^{\perp_1})$ where $\Scal$ denotes a set of syzygies of $T$, namely $\Scal=\{\Omega^m(T)\mid m\geq 0\}$ with $\Omega^0(T)=T$. Thus, $\Scal$ contains precisely the contractible projectives of the form $A^j$ and the representations of $\Rep(AQ_n/I)$ which have the regular $A$-module in one vertex and zero in all others.
Using \cite[Corollary 3.2.4]{GT}, it follows that the left hand side of the above cotorsion pair $^{\perp_1}(\Scal^{\perp_1})$ is given by all the representations in $\Rep(AQ_n/I)$ that are direct summands of $\Scal$-filtered objects. Therefore, we get that $^{\perp_1}(\Scal^{\perp_1})=\Rep_\mathscr{P}(AQ_n/I)$.

Dually, consider the $(n-1)$-cotilting module $C$ in $\Rep(AQ_n/I)$  given by the direct sum of all $(A^+)_j$ for $j\in\{-n+1,...,0\}$. The associated cotorsion pair in $\Rep(AQ_n/I)$ is of the form $(^\perp C,(^\perp C)^\perp)$ with $(^\perp C)^\perp=\Rep_\mathscr{I}(AQ_n/I)$. First observe that the class $^\perp C=^\perp\!\!\!(A^+)_0$ consists precisely of all representations $M$ in $\Rep(AQ_n/I)$ for which the cohomologies of $\Theta(M)$ vanish in all negative degrees. Thus, we get $\Rep_\mathscr{I}(AQ_n/I)\subseteq(^\perp C)^\perp$. For the converse, consider $M$ in $(^\perp C)^\perp$ and suppose that $M(0)$ is a non-injective $A$-module. Then we can use the injective envelope $E$ of $M(0)$ in $A\Mod$ to construct a non-trivial extension in $\Ext_{AQ_n/I}^1((E/M(0))_0,M)$ where $(E/M(0))_0$ clearly lies in $^\perp C$. Hence, $M(0)$ must be injective. Now we proceed inductively to prove the claim.
\end{example}

The next theorem identifies (co)silting complexes in $\D(A)$ with certain (co)tilting modules over $AQ_n/I$.

\begin{theorem}\label{silting is tilting}
The functor $\Psi:\Rep(AQ_n/I)\longrightarrow \D(A)$ induces a bijection between tilting modules over $AQ_n/I$ that lie in $\Rep_\mathscr{P}(AQ_n/I)$ and $n$-silting complexes in $\D(A)$, up to equivalence. Dually, $\Psi$ induces a bijection between cotilting modules over $AQ_n/I$ that lie in $\Rep_\mathscr{I}(AQ_n/I)$ and $n$-cosilting complexes in $\D(A)$, up to equivalence. 
\end{theorem}
\begin{proof}
We prove the silting case; the cosilting case is dual. First observe that the assignment is well-defined. Let $T$ be a tilting module of $\Rep(AQ_n/I)$ lying in $\Rep_\mathscr{P}(AQ_n/I)$. We show that $\Psi(T)$ is a silting complex in $\mathsf{D}(A)$. It follows from Lemma \ref{ff homological} that $\Hom_{\D(A)}(\Psi(T),\Psi(T)^{(J)}[i])=0$ for all $i>0$ and all sets $J$. Moreover, the coresolution of $AQ_n/I$ by $\mathsf{Add}(T)$ in $\Rep(AQ_n/I)$ shows that $\Psi(AQ_n/I)$ (which is isomorphic to $A$ in $\D(A)$) lies in $\mathsf{thick}(\Add(\Psi(T)))$ and, thus, $\mathsf{thick}(\Add(\Psi(T)))=\mathsf{K}^b(A\Proj)$.

We show that the assignment is injective. Indeed, let $T$ and $V$ be tilting modules in $\Rep(AQ_n/I)$ lying in $\Rep_\mathscr{P}(AQ_n/I)$ such that $\Psi(T)$ is equivalent to $\Psi(V)$. Suppose, without loss of generality, that there is an object $X$ in $\mathsf{Add}(T)$ that does not lie in $\mathsf{Add}(V)$. It is then clear that $\Psi(X)=0$ and, thus, $X$ is a contractible projective object in $\Rep_\mathscr{P}(AQ_n/I)$. However, it follows from Lemma \ref{properties CnA} that every such contractible projective, being both injective and projective in $\Rep_\mathscr{P}(AQ_n/I)$, must lie in $\mathsf{Add}(V)$, yielding a contradiction.

Finally, it remains to see that the assignment is surjective. By definition, any given $n$-silting complex lies in $\Psi(\Rep_\mathscr{P}(AQ_n/I))$. Let $T$ be an object in $\Rep_\mathscr{P}(AQ_n/I)$ representing it and assume, without loss of generality, that all the contractible projective objects lie in $\mathsf{Add}(T)$. We show that $T$ is an $(n-1)$-tilting module in $\Rep(AQ_n/I)$. It is clear from Lemma \ref{properties CnA} that $T$ has projective dimension at most $n-1$ and using Lemma \ref{ff homological} it follows that $\Ext^i_{AQ_n/I}(T,T^{(J)})=0$ for all $i>0$ and all sets $J$. It only remains to show the existence of an $\Add(T)$-coresolution for $AQ_n/I$. But this follows from \cite[Theorem 3.5]{Wei0}. Indeed, the coresolution of $A$ by $\Add(\Psi(T))$ in the derived category induces a coresolution of $A_0$ by $\Add(T)$ in $\Rep(AQ_n/I)$, since such triangles in $\K(A\Proj)$ induce short exact sequences in $\Rep_\mathscr{P}(AQ_n/I)$ once we add suitable contractible projective objects. It then remains to observe that the further summands of $AQ_n/I$ are contractible projective objects that already lie in $\Add(T)$.
\end{proof}

\begin{remark}
If the ring $A$ has finitistic dimension zero, then every tilting (respectively, cotilting) module in $\Rep(AQ_n/I)$ must already belong to $\Rep_\mathscr{P}(AQ_n/I)$ (respectively, to $\Rep_\mathscr{I}(AQ_n/I)$). In this case, $\Psi$ induces a bijection between all tilting (respectively, cotilting) modules over $AQ_n/I$ and all $n$-silting (respectively, $n$-cosilting) complexes in $\D(A)$, up to equivalence.
\end{remark}

The theorem above introduces the possibility of using results from tilting theory to prove statements about silting complexes. An easy application is to show that any partial $n$-(co)silting complex (suitably defined) admits a complement and can be completed to an $n$-(co)silting complex. This uses the well-known theory of complements for partial (co)tilting modules developed in \cite{AC}. In a more general setting, the existence of complements for partial silting objects will be discussed in \cite{APV}. In the next section, however, we focus on using the theorem above towards our aim of characterising silting and cosilting classes.


\section{Silting and Cosilting Classes}

\subsection{Torsion pairs in derived categories} We begin with a quick recollection of the necessary concepts.

\begin{definition}\label{def tp}
A pair of subcategories $(\Ucal,\Vcal)$ in $\D(A)$ is said to be a \textbf{torsion pair} if
\begin{enumerate}
\item $\Ucal$ and $\Vcal$ are closed under summands;
\item $\Hom_{\D(A)}(\Ucal,\Vcal)=0$;
\item For every object $X$ of $\D(A)$, there are $U$ in $\Ucal$, $V$ in $\Vcal$ and a triangle
$$U\longrightarrow X\longrightarrow V\longrightarrow U[1].$$
\end{enumerate}
In a torsion pair $(\Ucal,\Vcal)$, the class $\Ucal$ is said to be the \textbf{aisle}, the class $\Vcal$ the \textbf{coaisle}, and $(\Ucal,\Vcal)$ is called
\begin{itemize}
\item a \textbf{t-structure} if $\Ucal[1]\subseteq \Ucal$;
\item a \textbf{co-t-structure} if $\Ucal[-1]\subseteq \Ucal$;
\item \textbf{generated by a set of objects $\Scal$} if $(\Ucal,\Vcal)=({}^{\perp_0}(\Scal^{\perp_0}),\Scal^{\perp_0})$; 
\item \textbf{compactly generated} if $(\Ucal,\Vcal)$ is generated by a set of compact objects.
\end{itemize}
Finally, we say that a triple $(\Ucal,\Vcal,\Wcal)$ is a \textbf{torsion-torsionfree triple} (TTF triple for short) if both $(\Ucal,\Vcal)$ and $(\Vcal,\Wcal)$ are torsion pairs. In that case we say that $\Vcal$ is a \textbf{TTF class}, $(\Ucal,\Vcal)$ is \textbf{left adjacent} to $(\Vcal,\Wcal)$ and $(\Vcal,\Wcal)$ is \textbf{right adjacent} to $(\Ucal,\Vcal)$.
\end{definition}

There is a well-known \textbf{standard t-structure} in $\D(A)$ that we denote by $(\D^{\leq 0},\D^{\geq 1})$ and, as it is common in the literature, we write $\D^{\leq n}:=\D^{\leq 0}[-n]$ and $\D^{\geq n}:=\D^{\geq 1}[-n+1]$. The class $\D^{\leq 0}$ (respectively, $\D^{\geq 1}$) consists of all objects with cohomologies vanishing in positive (respectively, in non-positive) degrees. Moreover, this t-structure has both a left and a right adjacent co-t-structure, turning both $\D^{\leq 0}$ and $\D^{\geq 1}$ into TTF classes. For example, the class $\K_{\geq 1}:={}^{\perp_0}(\D^{\leq 0})$ consists of all complexes of projective $A$-modules whose non-positive components are zero. The co-t-structure $(\K_{\geq 1},\D^{\leq 0})$ will be referred to as the \textbf{standard co-t-structure} (see \cite[Example 2.9]{AMV1} and \cite[Example 2.4(1)]{AMV3} for further details).

\begin{definition}
We say that a subcategory $\Vcal$ of $\D(A)$ is \textbf{suspended} (respectively, \textbf{cosuspended}) if $\Vcal$ is closed under extensions and  $\Vcal[1]\subseteq \Vcal$ (respectively, $\Vcal[-1]\subseteq \Vcal$). Given a suspended (respectively, cosuspended) subcategory $\Vcal$, we say that $\Vcal$ is \textbf{intermediate} (respectively,  \textbf{co-intermediate}) if there are integers $n\leq m$ such that $\D^{\leq n}\subseteq \Vcal\subseteq \D^{\leq m}$ (respectively, $\D^{\geq m}\subseteq \Vcal\subseteq \D^{\geq n}$).
\end{definition}

\begin{example}\label{ex TTF}
Given a set $\Scal$ of compact objects in $\D(A)$, the pair $({}^{\perp_0}(\Scal^{\perp_0}), \Scal^{\perp_0})$ is always a torsion pair (see \cite[Theorem 4.3]{AI}) and, in fact, it follows from \cite[Theorem 3.11]{PS} that $\Scal^{\perp_0}$ is a TTF class. In particular, subcategories of the form $\Scal^{\perp_{>0}}=\{S[n]\mid S\in\Scal, n< 0\}^{\perp_0}$ are examples of suspended TTF classes.
\end{example}

In this section, we aim to extend the characterisation of tilting and cotilting classes in module categories, as recalled in the introduction, to the context of silting and cosilting classes in derived categories. The latter are defined as follows.

\begin{definition}
If $T$ is a silting complex in $\D(A)$, then we say that $T^{\perp_{>0}}$ is a \textbf{silting class}. If $C$ is a cosilting complex in $\D(A)$, then we say that ${}^{\perp_{>0}}C$ is a \textbf{cosilting class}.
\end{definition}

It is well known that a silting (respectively, cosilting) class determines a unique silting (respectively, cosilting) complex up to equivalence (see \cite[Theorem 5.3]{Wei0} and \cite[Theorem 2.17]{WeiZhang} for details).


\subsection{The silting case}
Recall from \cite{AMV1} that every silting class is the aisle of a t-structure and the coaisle of a co-t-structure or, in other words, that silting classes are suspended TTF classes. 

\begin{theorem}\cite[Theorem 4.6]{AMV1}\label{bij AMV1}
The assignment sending a silting complex $T$ in $\D(A)$ to its silting class $T^{\perp_{>0}}$ yields a bijection between silting complexes up to equivalence and intermediate suspended TTF classes.
\end{theorem}

The following result shows that the co-t-structures in $\D(A)$ associated to silting complexes are precisely those which are compactly generated and have an intermediate coaisle. This constitutes the first half of our main theorem and it generalises the finite-type characterisation of tilting classes in module categories.

\begin{theorem}\label{thm silting class}
A subcategory $\Vcal$ of $\D(A)$ is a silting class if and only if $\Vcal$ is intermediate and $\Vcal=\Scal^{\perp_{>0}}$ for a set $\Scal$ of compact objects. Moreover, the assignment $T\mapsto {}^{\perp_0}(T^{\perp_{>0}})\cap \K^b(A\proj)$ yields a bijection between silting complexes up to equivalence and cosuspended subcategories $\mathcal{R}$ of $\K^b(A\proj)$ which are co-intermediate in the sense that  there are integers $n\leq m$ such that $\K_{\geq m}\cap \K^b(A\proj)\subseteq \mathcal{R} \subseteq \K_{\geq n}\cap \K^b(A\proj)$.
\end{theorem}
\begin{proof}
Suppose that $\Vcal$ is intermediate and that $\Vcal=\Scal^{\perp_{>0}}$ for a set $\Scal$ of compact objects. It then follows from Example \ref{ex TTF} that $\Vcal$ is a suspended TTF class and, therefore, by the theorem above, $\Vcal$ is a silting class.

Conversely, suppose that $\Vcal$ is a silting class. Without loss of generality, we assume that the associated silting complex is $n$-silting, i.e. it lies in $\Psi(\Rep_\mathscr{P}(AQ_n/I))$. From the proof of Theorem \ref{bij AMV1}, it then follows that the class $\Vcal$ is intermediate with $\D^{\leq -n+1}\subseteq\Vcal\subseteq\D^{\leq 0}$. Moreover, by Theorem \ref{silting is tilting}, there is a tilting $AQ_n/I$-module $T$ in $\Rep_\mathscr{P}(AQ_n/I)$ representing our silting complex. From the main theorem in \cite{BS} we infer that the tilting class $T^{\perp}$ in $\Rep(AQ_n/I)$ is of finite type, that is, there is a set $\mathscr{S}$ of compact modules over $AQ_n/I$ (i.e. modules admitting a finite resolution by finitely generated projective $AQ_n/I$-modules) such that $\mathscr{S}^{\perp}=T^{\perp}$. In fact, we can choose $\mathscr{S}$ to contain all compact modules from $^\perp(T^{\perp})$. Since $T$ lies in $\Rep_\mathscr{P}(AQ_n/I)$, by Example \ref{example proj inj}, it follows that also $\mathscr{S}$ lies in $\Rep_\mathscr{P}(AQ_n/I)$. Therefore, we can define the set $\Scal:=\Psi(\mathscr{S})$ of compact objects in $\D(A)$ for which $\Scal^{\perp_{>0}}=\Psi(T)^{\perp_{>0}}=\Vcal$, by Lemma \ref{ff homological}. Indeed, by construction, the stalk complex with the regular module $A$ in degree zero belongs to $\Scal$, guaranteeing that $\Scal^{\perp_{>0}}$ is contained in $\D^{\leq 0}$, and that $\Scal^{\perp_{>0}}$ is intermediate as a consequence. Hence, computing the orthogonal classes $\Scal^{\perp_{>0}}$ and $\Psi(T)^{\perp_{>0}}$ in $\D(A)$ boils down to computing the corresponding Ext-orthogonal classes $\mathscr{S}^\perp$ and $T^\perp$ in $\Rep(AQ_n/I)$, which coincide by construction.

The final statement follows from the fact that a compactly generated co-t-structure is determined by the intersection of its aisle with the compact objects (see \cite[Theorem 4.5(ii)]{PS}).
\end{proof}

\begin{remark}
\begin{enumerate}
\item The last statement of Theorem \ref{thm silting class} generalises the correspondence in module categories between tilting modules up to equivalence and resolving subcategories of compact modules of bounded projective dimension (see \cite[Theorem 2.2]{AHT}). Moreover, the proof above identifies the tilting cotorsion pair in $\Rep(AQ_n/I)$ generated by such a resolving subcategory with the silting co-t-structure in the derived category.
\item Note that Theorem \ref{thm silting class} (together with Theorem \ref{bij AMV1}) can be interpreted as a variant of the telescope conjecture (see \cite[Question 3.12]{PS} for an explanation of the \textit{unstable telescope conjecture}). More concretely, it is shown that a co-t-structure in $\D(A)$ with an intermediate coaisle admits a right adjacent t-structure if and only if it is compactly generated.
\end{enumerate}
\end{remark}


\subsection{The cosilting case}
We begin by recalling from \cite{Bel} and \cite{Kr1} some ideas concerning purity in derived categories.
Recall that a morphism $f:X\rightarrow Y$ in $\D(A)$ is said to be a \textbf{pure monomorphism} if $\Hom_{\D(A)}(K,f)$ is a monomorphism for any compact object $K$. An object $C$ is said to be \textbf{pure-injective} if any pure monomorphism $C\rightarrow Y$ splits. The following theorem provides a useful description of pure-injective objects.

\begin{theorem}\cite[Theorem 1.8]{Kr1}\label{Krause}
The following statements are equivalent for an object $C$ in $\D(A)$.
\begin{enumerate}
\item $C$ is pure-injective;
\item For every set $J$, the summation map $C^{(J)}\rightarrow C$ factors through the canonical map $C^{(J)}\rightarrow C^J$.
\end{enumerate}
\end{theorem}

There is an analogous notion of pure-injectivity for categories of modules, and the theorem above also holds in this context (see \cite[Chapter 7]{JL}). Pure-injective objects are intimately related with definable subcategories. We recall the relevant definition.

\begin{definition}
A subcategory $\Vcal$ of $\D(A)$ is said to be \textbf{definable} if there is a set of maps $(\phi_j:X_j\longrightarrow Y_j)_{j\in J}$ in $\mathsf{K}^b(A\proj)$ such that an object $V$ lies in $\Vcal$ if and only if $\Hom_{\D(A)}(\phi_j,V)$ is surjective for all $j$ in $J$.
\end{definition}

Definable subcategories of $\D(A)$ were shown to be preenveloping in \cite{AMV3} and they are well known to satisfy several closure properties, among which the closure under products and coproducts. We will see that cosilting classes are always definable and that they give rise to t-structures with nice homological properties.

\begin{proposition}\label{cor pure}
Let $C$ be a cosilting complex in $\D(A)$. Then $C$ is pure-injective and the cosilting class ${}^{\perp_{>0}}C$ is definable. Moreover, the pair $({}^{\perp_{\leq 0}}C,{}^{\perp_{>0}}C)$ is a t-structure and its heart $\Hcal:={}^{\perp_{\leq 0}}C[-1]\cap {}^{\perp_{>0}}C$ is a Grothendieck abelian category.
\end{proposition}
\begin{proof}
Let $C$ be a cosilting complex in $\mathsf{D}(A)$. We show that it is pure-injective by using Theorem \ref{Krause}. Without loss of generality, by Theorem \ref{silting is tilting}, there is an integer $n>0$ and an $(n-1)$-cotilting module $M$ over $AQ_n/I$ such that $\Psi(M)=C$. Let $J$ be a set and $f_J:M^{(J)}\rightarrow M$ (respectively, $g_J:M^{(J)}\rightarrow M^J$) be the summation map (respectively, the canonical map from the coproduct to the product) of $M$ in $\Rep(AQ_n/I)$. Recall that the maps $f_J$ and $g_J$ are uniquely determined by the universal properties of the product $M^{J}$ and the coproduct $M^{(J)}$. Since $\Psi$ preserves products and coproducts (which are exact and defined componentwise), it follows from the universality of these maps that $\Psi(f_J)$ is the summation map $C^{(J)}\rightarrow C$ and $\Psi(g_J)$ is the canonical map $C^{(J)}\rightarrow C^J$ in $\D(A)$. Since $M$ is a cotilting module, it is pure-injective in $\Rep(AQ_n/I)$ by \cite{Stov} and it yields a factorisation of $f_J$ through $g_J$. The image of this factorisation under $\Psi$ shows that $C$ is pure-injective in $\D(A)$. 
Moreover, using that ${}^{\perp_{>0}}C$ is closed under products (\cite[Proposition 2.12]{WeiZhang}), it follows from \cite[Lemma 4.8]{AMV3} that the subcategory ${}^{\perp_{>0}}C$ is definable and that the pair $({}^{\perp_0}({}^{\perp_{>0}}C),{}^{\perp_{>0}}C)$ is a t-structure in $\D(A)$. Now, the same arguments as in the proof of (1) $\Rightarrow$ (2) in \cite[Theorem 4.9]{AMV3} show that ${}^{\perp_0}({}^{\perp_{>0}}C)={}^{\perp_{\leq 0}}C$. Finally, it only remains to refer to \cite[Theorem 3.6]{AMV3} to conclude that the heart $\Hcal$ of this t-structure is a Grothendieck abelian category.
\end{proof}

\begin{remark}
\begin{enumerate}
\item In \cite[page 56]{BR}, cotilting objects in a compactly generated triangulated category are required by definition to be pure-injective. 
The proposition above shows that for any cotilting complex in $\D(A)$ such an assumption is superfluous.
\item The approximation triangles for the t-structure $({}^{\perp_{\leq 0}}C,{}^{\perp_{>0}}C)$ coming from a cosilting complex $C$ can be constructed explicitly. Indeed, considering successive left approximations to suitable shifts of $\mathsf{Prod}(C)$, we can build an inverse system, the Milnor limit of which is an object of ${}^{\perp_{\leq 0}}C$. Moreover, the cone of the resulting map lies in the smallest cosuspended and product-closed subcategory of $\D(A)$ containing $C$, which turns out to be ${}^{\perp_{>0}}C$. This argument was explained to us by Jan \v{S}\v{t}ovi\v{c}ek.
\end{enumerate}
\end{remark}

It is a well-known fact from tilting theory that cotilting classes in a module category are not always of cofinite type, i.e. they are generally not determined by a set of compact modules (see \cite{Ba}). The following example shows that cosilting t-structures in the derived category of a ring are not always compactly generated. Note that this example provides an answer to \cite[Question 3.5]{BP}.

\begin{example}
We build on \cite[Example 5.4]{AH}. Let $A$ be a commutative local ring whose maximal ideal $\mathfrak{m}$ is non-zero and idempotent. Then $A/\mathfrak{m}$ is a cosilting $A$-module yielding a torsion pair $(\Tcal,\Fcal:=\Add(A/\mathfrak{m}))$ in $A\Mod$. Let $C$ be the $2$-cosilting complex corresponding to the cosilting module $A/\mathfrak{m}$, that is, $C$ is obtained from $A/\mathfrak{m}$ by passing to a suitable injective copresentation, as discussed in \cite{AH}. The associated cosilting t-structure $({}^{\perp_{\leq 0}}C,{}^{\perp_{>0}}C)$ in $\D(A)$ arises as an HRS-tilt of the torsion pair $(\Tcal,\Fcal)$. By \cite[Theorem 2.3]{BP}, it follows that this t-structure is compactly generated if and only if $\Fcal$ is the Hom-orthogonal class to a set of finitely presented $A$-modules. But the latter would imply, following \cite[Lemma 3.7 and Lemma 4.2]{AH}, that $(\Tcal,\Fcal)$ is a hereditary torsion pair, which clearly is not the case ($\Fcal$ is not closed under injective envelopes).
\end{example}

The following result is  the expected dual statement of Theorem \ref{bij AMV1} for cosilting complexes. It can now be obtained from Proposition \ref{cor pure} together with recent results in \cite{Bondarko} and \cite{WeiZhang}.

\begin{theorem}\label{dual TTF}
The assignment sending a cosilting complex $C$ in $\D(A)$ to its cosilting class ${}^{\perp_{>0}}C$ yields a bijection between cosilting complexes up to equivalence and co-intermediate cosuspended TTF classes.
\end{theorem}
\begin{proof}
First, observe that the assignment is well-defined. Indeed, the cosilting class ${}^{\perp_{> 0}}C$ depends only on the equivalence class of $C$, since we have ${}^{\perp_{> 0}}C={}^{\perp_{> 0}}\mathsf{Prod}(C)$. Moreover, it is clear that, since $C$ is a bounded complex of injective $A$-modules, there is an integer $m$ such that $\D^{\geq m}\subseteq {}^{\perp_{>0}}C$. It then follows from \cite[Proposition 2.10]{WeiZhang} that cosilting classes are co-intermediate. We need to check that the t-structure $({}^{\perp_{\leq 0}}C, {}^{\perp_{> 0}}C)$ admits a right adjacent co-t-structure. For this purpose, we use \cite[Corollary 3.2.6]{Bondarko}. Recall that every object $X$ in $\D(A)$ can be built as a Milnor limit of its standard truncations 
(see \cite[Remark 2.3]{BN}). Hence, it follows that the smallest colocalising subcategory of $\D(A)$ containing ${}^{\perp_{>0}}C$ is $\D(A)$ itself, which is well known to satisfy Brown representability. 
This condition, together with the fact that the heart of $({}^{\perp_{\leq 0}}C, {}^{\perp_{> 0}}C)$ has enough injectives by Proposition \ref{cor pure}, yields the existence of a right adjacent co-t-structure by \cite[Corollary 3.2.6]{Bondarko}.

Now, since every co-intermediate cosuspended TTF class satisfies the characterisation of cosilting classes in \cite[Theorem 2.17]{WeiZhang}, it follows that the assignment is surjective (the arguments are dual to those in the proof of \cite[Theorem 4.6]{AMV1}). Finally, note that the assignment is injective, since the cosilting class determines the cosilting complex up to equivalence. 
\end{proof}

Finally, we can now prove the second half of our main theorem, as stated in the introduction.

\begin{theorem}\label{thm cosilting class}
A subcategory $\Vcal$ of $\D(A)$ is a cosilting class if and only if $\Vcal$ is cosuspended, co-intermediate and definable.
\end{theorem}
\begin{proof}
By combining Proposition \ref{cor pure} with Theorem \ref{dual TTF}, it only remains to show that if $\Vcal$ is cosuspended, co-intermediate and definable, then it is a TTF class. First, it follows from \cite[Proposition 4.5]{AMV3} that $({}^{\perp_0}\Vcal,\Vcal)$ is a t-structure. Moreover, since $\Vcal$ is co-intermediate, it contains some shift of the standard coaisle (i.e. there is some integer $m$ such that $\D^{\geq m}\subseteq \Vcal$). Hence, it follows again from \cite[Corollary 3.2.6]{Bondarko}, as argued in Theorem \ref{dual TTF}, that $(\Vcal,\Vcal^{\perp_0})$ is a co-t-structure.
\end{proof}


\end{document}